\documentclass[12pt,reqno]{amsart}
\usepackage[leqno]{amsmath}
\usepackage[utf8]{inputenc}
\usepackage{amsmath}
\usepackage{amsfonts}
\usepackage{amssymb}
\usepackage{fullpage}
\usepackage{upref}
\usepackage{hyperref}
\usepackage{amsthm}
\usepackage{fancyhdr,color}
\newtheorem{theorem}{Theorem}[section]
\theoremstyle{plain}

\newtheorem{lemma}[theorem]{Lemma}
\newtheorem{proposition}[theorem]{Proposition}

\newtheorem{definition} {Definition}
\theoremstyle{definition}

\renewcommand{\d}{\/\mathrm{d}\/}

\def\S{\mathrm{S}}

\def\B{\mathcal{B}}
\def\K{\mathcal{K}}
\def\D{\mathrm{D}}

\def\no{\nonumber}

\def\u{\mathbf{u}}

\renewcommand{\d}{\/\mathrm{d}\/}

\makeatletter
\newcommand{\leqnomode}{\tagsleft@true}
\newcommand{\reqnomode}{\tagsleft@false}

\newcommand{\Addresses}{{
		\footnote{

			\noindent \textsuperscript{1}Dipartimento di Matematica, Universita` degli Studi di Pavia and IMATI-C.N.R., Via Ferrata 5, 27100 Pavia, Italy.  \par\nopagebreak \noindent
			\textit{e-mail:} \texttt{tania.biswas@unipv.it}

			\noindent \textsuperscript{2}Dipartimento di Matematica, Universita` degli Studi di Pavia and IMATI-C.N.R., Via Ferrata 5, 27100 Pavia, Italy.  \par\nopagebreak \noindent
			\textit{e-mail:} \texttt{elisabetta.rocca@unipv.it}

			\noindent \textsuperscript{*}Corresponding author.

			\medskip\noindent
			{\bf Acknowledgments:}  Tania Biswas  would like to thank Department of Mathematics, University of Pavia for providing financial support and stimulating environment for the research and Prof. Elisabetta Rocca for fruitful discussions. This research was supported by the Italian Ministry of Education, University and Research (MIUR): Dipartimenti di Eccellenza Program (2018–2022) – Dept. of Mathematics “F. Casorati”, University of Pavia. In addition, this research has been performed in the framework of the project Fondazione Cariplo-Regione Lombardia  MEGAsTAR ``Matematica d'Eccellenza in biologia ed ingegneria come acceleratore di una nuova strateGia per l'ATtRattivit\`a dell'ateneo pavese''. The present paper also benefits from the support of the GNAMPA (Gruppo Nazionale per l'Analisi Matematica, la Probabilit\`a e le loro Applicazioni) of INdAM (Istituto Nazionale di Alta Matematica).}}}

\begin{document}
	
	\title[Long time dynamics]{Long time dynamics of a phase-field model of prostate cancer growth with chemotherapy and antiangiogenic therapy effects	\Addresses	}

	\author[Tania Biswas and Elisabetta Rocca]
	{Tania Biswas\textsuperscript{1*}, Elisabetta Rocca\textsuperscript{2}} 
	\maketitle

	\begin{abstract}
		We consider a phase-field model of prostate cancer growth with chemotherapy and antiangiogenic therapy effects which is introduced in \cite{LG}. It is comprised of  phase-field equation to describe tumor growth, which is coupled to a reaction-diffusion type equation for generic nutrient for the tumor. An additional equation couples the concentration of  prostate-specific antigen (PSA) in the prostatic tissue and it obeys a linear reaction-diffusion equation. The system completes with  homogeneous Dirichlet boundary conditions for the tumor variable and Neuman boundary condition for the nutrient and the concentration of PSA.  Here we investigate the long time dynamics of the model. We first prove that the initial-boundary value problem generates a strongly continuous semigroup on a suitable phase space that admits the global attractor in a proper phase space.  Moreover, we also discuss the convergence of a solution to a single stationary state and obtain a convergence rate estimate under some conditions on the coefficients.
\end{abstract}

	

\section{Introduction}
Recently, mathematical modeling and simulation of tumor growth have turned out to be promising in cancer research and have gained a lot of interest. Prostate cancer (PCa) is the second most common cancer among men worldwide according to WHO (World Health Organization). Mostly PCa is an adenocarcinoma, that is a form of cancer that starts growing in the epithelial tissue of the prostate. In general the structure, behavior, and evolution of a tumor depend on the genetic alterations that initiate it and microenvironmental conditions of the tumor. Tumor-induced angiogenesis i.e. the growth of new blood vessels from pre-existing ones via chemical signals produced by the tumor, is a key process that can control the microenvironmental conditions of the tumor. Moreover if the most effective ways i.e. combination of prevention and regular screening for early detection to tackle the growth of tumor do not work, then chemotherapy plays an efficient role. Chemotherapy for PCa is mostly based on cytotoxic drugs (e.g., docetaxel, cabazitaxel), which slows down tumor growth in terms of inhibiting cell proliferation and promoting tumor cell death. On the other hand outside of the standard care for PCa, antiangiogenic therapies have been actively investigated. However antiangiogenic drug is counterintuitive when administered alone, but it showed significant benefit to the patients when combined with chemotherapy. 

In \cite{LG}, the authors have proposed a model of PCa growth and chemotherapy, where the critical nutrient controls the advancement of the tumor. The effect of chemotherapy and antiangiogenic drugs are incorporated by downregulating tumor net proliferation and reducing nutrient supply, respectively. The model also exhibits an output of the time evolution of serum PSA, which  represents as monitoring tumor evolution or the disease’s response to a particular treatment in the clinical practice. The PDE that describes the evaluation of concentration of PSA is a linear reaction-diffusion equation. The ones that are representing the evaluation of the tumor phase $\phi$ is an Allen-Cahn type equation and  the nutrient properties $\sigma$ is a a reaction diffusion equation.

The model is given by 
\begin{subequations}
		\begin{align}
		\phi_t -\lambda \Delta \phi + 2\phi (1-\phi) f(\phi,\sigma,u) &= 0, \label{phi}\ \text{ in }\ Q_T:=(0,T)\times\Omega,\\
				\sigma_t - \eta \Delta \sigma + \gamma_h \sigma + (\gamma_c-\gamma_h) \sigma \phi &= S_h + (S_c-S_h)\phi - s\phi, \label{sigma}\ \text{ in }\ Q_T,\\
		p_t - D \Delta p + \gamma_p p  &= \alpha_h + (\alpha_c-\alpha_h)\phi, \label{p}\ \text{ in }\ \Omega\times(0,T),\\
		\phi = 0, \ \frac{\partial \sigma}{\partial \nu} &= \frac{\partial p}{\partial \nu} = 0, \text{ on }\ \Sigma_T:=(0,T)\times\partial\Omega, \label{boundary}\\
		\phi(0) &= \phi_0, \  \ \sigma(0) = \sigma _0, \ p(0)=p_0, \text{ in } \ \Omega, \label{initial}
		\end{align}   
	\end{subequations}
where $\Omega$ is an open bounded subset of $\mathbb{R}^N, \ N\leq 3,$ with a sufficiently smooth boundary $\partial \Omega,$ $\frac{\partial}{\partial \nu}$ represents the outward normal derivative to $\partial \Omega$ and $T$ denote final time. Here, the phase field is denoted by $\phi$ and it takes the transitions from the value $\phi=0$ in the host tissue to $\phi=1$ in the tumor. The transition is smooth but steep and takes on a hyperbolic tangent profile in the direction perpendicular to the interface. In \eqref{phi} the nonlinearity $f$ takes the form $$f(\phi,\sigma,u) = M[1-2\phi-3(m(\sigma)-m_{ref}u)]$$ where $M>0$ is a real constant which denotes the tumor mobility. The function $m(\sigma) $ is usually called tilting function and in this model is defined by $$m(\sigma) = m_{ref} \left( \frac{\rho+A}{2} + \frac{\rho-A}{\pi}\text{arctan} \left( \frac{\sigma - \sigma_l}{\sigma_r}\right)  \right),$$ where $\rho$ and $A$ denote constant proliferation and apoptosis indices, respectively. The positive constants $\sigma_r$and $\sigma_l$ represents a reference and a threshold value for the nutrient concentration respectively. The function $u$ denotes the tumor-inhibiting effect of a cytotoxic drug. In the equation \eqref{sigma}, the positive constants $\eta$, $S_h$, $S_c$, denote the diffusion coefficient of the nutrient, the nutrient supply rate in the healthy tissue, the nutrient supply rate in the cancerous tissue respectively and $s$ is a given function yielding the reduction in nutrient supply caused by antiangiogenic therapy. Moreover $\gamma_h$, $\gamma_c$ are positive constants that represent the nutrient uptake rate in the healthy and cancerous tissue, respectively. In the equation \eqref{p}, the positive constants $\gamma_h$, $D$, $\alpha_h$ and $\alpha_c$ denote the decay rate in which the PSA is diffused through the prostatic tissue, the diffusion constant, the tissue PSA production rate of healthy and malignant cells respectively. In \cite{LG}, the authors have studied well-posedness of \eqref{phi}-\eqref{initial} and developed an algorithm to solve the equations numerically. They have shown computationally that complex tumor dynamics matches well with the previous observations in both computational and clinical studies. For more details regarding the model refer \cite{LG},  \cite{LG1} and references therein.

Now we mention briefly a short overview of mathematical literature regarding the tumor growth models. In the recent past, phase field method to model tumor growth have turned out to be promising and have gained lot of interest. Various tumor growth models of Allen-Cahn type and Cahn-Hilliard type have been investigated for example in \cite{CGH}, \cite{DRSS}, \cite{FGR}, \cite{CGRS}, \cite{XG}. Later, including the mechanical effects into the tumor growth model in terms of Darcy's law or Brinkman's equation have been studied for example in \cite{GL}, \cite{GLSS}, \cite{JWZ}, \cite{LTZ}, \cite{EG}. A further class of tumor growth model that also include chemotaxis, transport effects and angiogenesis have been subsequently introduced for example in \cite{GL1}, \cite{GLSS}, \cite{XG}. Various optimal control problems related to tumor growth models are extensively addressed in for example \cite{CRW}, \cite{CGMR}, \cite{CGRS1}, \cite{EK}, \cite{GLR}.

Long time behavior for the tumor growth models has been one of the far-reaching areas of applied mathematics. Regarding the existing literature on the aspect of long time behavior, we mention \cite{SFR}, \cite{CGRS}, \cite{FRP}, \cite{lauren}, \cite{HGZ} for phase field equations. Turning to the study of long time behavior for tumor growth model, only few contributions have been published so far. Among that we can quote \cite{CRW}, \cite{CGH}, \cite{JWZ}, \cite{MRS}, \cite{RS} for example.

In this paper we are interested to study the long time behavior of trajectories of the dynamical system associated to the model. We first show that the dynamical system associated to the model admits a global attractor. Next, we show that the $\omega$-limit set of each solution trajectory is not empty and only consists of steady state solutions $(\phi_\infty,\sigma_\infty,p_\infty)$. More precisely, we derive the steady state system corresponding to the model and we deduce that $\phi_\infty$ is equal to 0. Moreover, we prove that any global weak solution will converge to a steady state solution as $t \rightarrow \infty$ under suitable conditions on the coefficients and we also provide an estimate on the convergence rate.

The paper is organized as follows: In section 2, we discuss the mathematical preliminaries such as existence and uniqueness of weak solution related to the model. In section 3, we prove the existence of a global attractor in the suitable phase space and we prove the long time behavior results related to the phase-field model of prostate cancer growth in section 4.

\section{Functional Setting} 

Let us take $H=L^2(\Omega)$ and  $V=H^1(\Omega)$. The standard inner product and norm in $H$ are denoted by $(\cdot,\cdot)$ and $\| \cdot\|$ respectively. Since the embedding $V \subset H$ is dense and continuous, we identify $H$ with its dual space $H'$ with the above scalar product and we get the gelfand triplet $V \subset H \subset V'$. The duality pairing between a general banach space $X$ and its dual $X'$ will be denoted by $\langle \cdot, \cdot \rangle$. Let us set the following Sobolev spaces
\begin{align}
V_0=H^1_0(\Omega), \ V_0'=(H^1_0(\Omega))'=: H^{-1}_0(\Omega), \ V=H^1(\Omega), \ V'=(H^1(\Omega))', \no \\
W_0=H^2(\Omega) \cap H^1_0(\Omega) , \ W=\Big\{y\in H^2(\Omega); \frac{\partial y}{\partial \nu}=0 \text{  on  } \partial \Omega \Big\},
\end{align}
with the dense and compact injections $W_0 \subset V_0 \subset H \subset V_0'$ and $W \subset V \subset H \subset V'$. Here $H^1_0(\Omega)$ contains the elements of $H^1(\Omega)$ with null trace on the boundary $\partial \Omega$.

Let $A:V_0 \to V_0$ be the laplacian operator with dirichlet boundary condition. It should be noted that  $A^{-1} : H \to H$ is a self-adjoint compact operator on $H$ and by the classical \emph{spectral theorem}, there exists a sequence $\{\lambda_j\}$ with $0<\lambda_1\leq \lambda_2\leq \lambda_j\leq\cdots\to+\infty$
and a family of $\mathbf{e}_j \in \D(A)$ which is orthonormal in $H$ and such that $A\mathbf{e}_j =\lambda_j\mathbf{e}_j$. We know that $\u \in\D(A)$ can be expressed as $u=\sum\limits_{j=1}^{\infty}\langle \u,\mathbf{e}_j\rangle \mathbf{e}_j,$ so that $Au=\sum\limits_{j=1}^{\infty}\lambda_j\langle u,\mathbf{e}_j\rangle \mathbf{e}_j$. Thus, it is immediate that 
\begin{align}\label{lambda}
\|\nabla u\|^2=\langle A u,u\rangle =\sum_{j=1}^{\infty}\lambda_j|\langle u,\mathbf{e}_j\rangle|^2\geq \lambda_1\sum_{j=1}^{\infty}|\langle u,\mathbf{e}_j\rangle|^2=\lambda_1\|u\|^2.
\end{align}

Let us take the following assumptions
\begin{align}\label{ass 1}
u \in L^{\infty}(Q_T) \text{  and   } s \in L^{\infty}(Q_T) \text{ and } s\geq 0, \ \text{ a.e.  in } Q_T .
\end{align}
Assume moreever that the parameters $\lambda,\ \eta, \ D, \gamma_c, \ \gamma_h, \ \gamma_p, \  \alpha_c, \ \alpha_h, \  S_c, \ S_h$ are positive.
We first state the existence and uniqueness theorem which is proved in \cite{LG}. For that let us introduce the spaces
\begin{align*}
X_0 &= W^{1,2}(0,T;V_0') \cap C([0,T];H) \cap L^2(0,T;V_0), \\
X &= W^{1,2}(0,T;V') \cap C([0,T];H) \cap L^2(0,T;V),
\end{align*}

\begin{align*}
\chi_0 &= W^{1,2}(0,T;H) \cap C([0,T];V_0) \cap L^2(0,T;W_0), \\
\chi &= W^{1,2}(0,T;H) \cap C([0,T];V) \cap L^2(0,T;W).
\end{align*}

\begin{definition}\label{def1}
A solution to the system \eqref{phi}-\eqref{initial} is a triplet $(\phi,\sigma,p)$, with $\phi \in X_0 \cap L^{\infty}(Q_T), \ \sigma \in X, p \in X,$ which satisfies 
\begin{subequations}
		\begin{align}
&\int_0^T \langle \phi_t, \psi_1(t) \rangle_{V_0',V_0} \d t + \int_{Q_T} \{\lambda \nabla \phi \cdot \nabla \psi_1 + 2\phi (1-\phi) f(\phi,\sigma,u) \psi_1 \} \d x \d t =0, \label{phi1} \\
&\quad \int_0^T \langle\sigma_t, \psi_2(t) \rangle_{V',V} \d t + \int_{Q_T} \{\eta \nabla  \sigma \cdot \nabla \psi_2 + \gamma_h \sigma \psi_2 + \gamma_{ch} \sigma \phi \psi_2 \} \d x \d t,  \label{sigma1}  \\
&\quad = \int_{Q_T} \{S_h \psi_2 + (S_{ch}- s)\phi \psi_2\} \d x \d t, \text{ for all  } (\psi_1,\psi_2) \in L^2(0,T;V_0 \times V),  \no \\
&\int_0^T p_t, \psi(t) \rangle_{V',V} \d t  + \int_{Q_T} \{D \nabla p \cdot \nabla \psi + \gamma_p p \psi\} \d x \d t  \label{p1} \\ 
&\quad = \int_{Q_T} \{\alpha_h + \alpha_{ch} \phi\psi \} \d x \d t, \text{ for all  } \psi \in L^2(0,T;V) \no
\end{align}
\end{subequations}
and 
\begin{equation}\label{initial1}
(\phi,\sigma,p)(0) = (\phi_0,\sigma_0,p_0).
\end{equation}
\end{definition}
We recall now the results of \cite{LG}.
\begin{theorem}
Let 
\begin{align}
(\phi_0,\sigma_0,p_0) \in H \times H \times H, \\
0 \leq \phi_0 \leq 1 \ \text{ a.e. } x \in \Omega.
\end{align}
Then the system \eqref{phi}-\eqref{initial} has a unique solution $(\phi,\sigma,p)$ in the sense of Definition \ref{def1}, such that 
\begin{eqnarray}
0 \leq \phi(t,x) \leq 1 \ \text{ a.e. } (t,x) \in Q_T.
\end{eqnarray}
If $(\sigma_0,p_0) \in L^\infty(\Omega) \times L^\infty(\Omega),$ then $(\sigma,p) \in L^\infty(Q_T) \times L^\infty(Q_T).$ Moreover, if 
\begin{eqnarray}
\sigma_0(x) \geq 0, p_0(x) \geq 0 \ \text{ a.e. } x \in \Omega, \ \ s(t,x)  \leq S_c \ \text{ a.e. } (t,x) \in Q_T,
\end{eqnarray}
then we have 
\begin{eqnarray}
\sigma(t,x) \geq 0, p(t,x) \geq 0 \ \text{ a.e. } (t,x) \in Q_T.
\end{eqnarray}
If we take $(\phi_0,\sigma_0,p_0) \in V_0 \times V \times V,$ then the solution satisfies the following estimate 
\begin{align}
\|\phi\|_{X_0} + \|\sigma\|_{X} + \|p\|_{X} \leq C \left( \|\phi_0\|^2_{V_0} + \|\sigma_0\|^2_{V} + \|p_0\|^2_{V} + \|u\|^2_{L^2(0,T;H)} + \|s\|^2_{L^2(0,T;H)} +1 \right).
\end{align}
Moreover, the solution is continuous with respect to the data, that is, for two solutions $(\phi_i,\sigma_i,p_i)$ corresponding to the date $(\phi^i_0,\sigma^i_0,p^i_0,u_i,s_i),$ $i=1,2$, we have 
\begin{align}
&\|(\phi_1-\phi_2)\|_{H}^2 + \|(\sigma_1-\sigma_2)\|_{H}^2 + \|(p_1-p_2)\|_{H}^2  \no \\
&+ \|(\phi_1-\phi_2)\|_{L^2(0,T;V_0)}^2 + \|(\sigma_1-\sigma_2)\|_{L^2(0,T;V)}^2 + \|(p_1-p_2)\|_{L^2(0,T;V)}^2 \no \\
&\leq C \left( \|(\phi^1_0-\phi^2_0)\|^2_{H} + \|(\sigma^1_0-\sigma^2_0)\|_{H}^2 + \|(p^1_0-p^2_0)\|_{H}^2 + \|u_1-u_2\|^2_{L^2 (0,T;H)} + \|s_1-s_2\|^2_{L^2 (0,T;H)}  \right) \no,
\end{align}
for all $t \in [0,T].$
\end{theorem}

\section{Dissipativity}
We recall the statement of the so called uniform Gronwall's lemma, that we will need to prove dissipativity of the dynamical system associated to \eqref{phi1}-\eqref{initial1}.
\begin{lemma}\label{lem3.1}
Let $y,a,b \in L^1_{\text{loc}}(0,\infty)$ three non-negative functions such that $y'\in L^1_{\text{loc}}(0,\infty)$ and 
\begin{eqnarray}
y'(t) \leq a(t)y(t) + b(t) \text{ for a.e. } t>0,
\end{eqnarray}
and let $a_1, a_2, a_3$ three non-negative constants such that 
\begin{eqnarray}
\|a\|_{\tau^1(\mathbb{R})} \leq a_1, \quad \|b\|_{\tau^1(\mathbb{R})} \leq a_2, \quad \|y\|_{\tau^1(\mathbb{R})} \leq a_3.
\end{eqnarray}
Then, we have that 
\begin{eqnarray}
y(t+1) \leq (a_2+a_3)e^{a_1} \text{  for all } t>0.
\end{eqnarray}
\end{lemma}
Now let us recall some basic notions on absorbing sets and attractors. Given a strongly continuous semigroup $S(t)$ on a complete metric space $(X,d_X),$ we say that $\mathcal{B}_0$ is an absorbing set for $S(t)$ iff:
\begin{itemize}
\item $\mathcal{B}_0$ is bounded;
\item for any bounded set $\mathcal{B} \subset X,$ there exist a time $T_{\mathcal{B}}\geq 0$ such that 
\begin{eqnarray}
S(t)\mathcal{B} \subset \mathcal{B}_0 \quad \forall t \geq T_{\mathcal{B}}.
\end{eqnarray} 
\end{itemize}
Next, a set $\mathcal{K} \subset X$ is said tobe uniformly attracting for the semigroup $S(t)$ iff for any bounded set $\mathcal{B} \subset X$, we have 
\begin{eqnarray}
\lim_{t \rightarrow \infty} \partial(S(t)\mathcal{B},\mathcal{K}) =0,
\end{eqnarray}
where $\partial$ denotes the unilateral Hausdroff distance of the set $S(t)\mathcal{B}$ from $\mathcal{K}$, with respect to the metric of $X$, i.e. 
\begin{eqnarray}
\partial(S(t)\mathcal{B},\mathcal{K}) := \sup_{y \in S(t)\mathcal{B}} \inf_{k \in \mathcal{K}} d_X(y,k).
\end{eqnarray}
Finally, a set $\K$ is the universal attractor of the semigroup $S(t)$ iff:
\begin{itemize}
\item $\K$ is attracting and compact in $X$;
\item $\K$ is fully invariant with respect to $S(t)$, i.e. $S(t) \K=\K$ for all $t\geq 0$. 
\end{itemize}
Also from (Section I.1.3, \cite{temam}) we remark that if the universal attractor exists, it is a unique and connected set. Now we state a general abstract criterion (Theorem I.1.1, \cite{temam}) which provides a sufficient condition for the existence of the attractor.
\begin{theorem}
Let $S(t)$ be a strongly continuous semigroup on the complete metric space $(X,d_X)$. Let us assume that:
\begin{itemize}
\item $S(t)$ admits an absorbing set $\B_0$ (dissipativity);
\item for any bounded set $\B \subset X$, there exist $t_{\B}>0$ such that 
\begin{eqnarray}
\cup_{t \geq t_{\B}} S(t) \B \text{  is compact in  }X \text{ (uniform compactness)}.
\end{eqnarray}
Then, $S(t)$ admits the universal attractor $\mathcal{K}$ which is given by 
\begin{eqnarray}
\mathcal{K}=\omega - \lim(S(t) \B_0)=\cap_{\tau \geq 0} \overline{\cup_{t \geq \tau} S(t) \B_0} 
\end{eqnarray}
\end{itemize}
\end{theorem}

We first derive some boundedness property for the nutrient $\sigma$. Let us denote $\tilde{\sigma}=\max\{\frac{\S_h}{\gamma_h},\frac{\S_c}{\gamma_c}\}$ and $\sigma_0 \leq \tilde{\sigma}.$
\begin{theorem}\label{sigmabound}
Let the assumption \ref{ass 1} holds and assume that $(\phi_0,\sigma_0,p_0) \in H \times H \times H,$ with $0 \leq \sigma_0 \leq \tilde{\sigma}$ a.e. $x\in \Omega$ where $\tilde{\sigma}$ is $\max\{\frac{\S_h}{\gamma_h},\frac{\S_c}{\gamma_c}\}$. Then the system \eqref{phi}-\eqref{initial} has a unique solution $(\phi,\sigma,p)$ in the sense of Definition \ref{def1}, such that 
\begin{eqnarray}
0 \leq \sigma \leq \tilde{\sigma} \text{  a.e.  } (t,x) \in Q_T.
\end{eqnarray}
\end{theorem}
\begin{proof}
For this we test \eqref{sigma1} by $(\sigma - \tilde{\sigma})^+$ and deduce
\begin{align}
\frac{1}{2}\|(\sigma - \tilde{\sigma})^+(t)\|^2 &+ \int_0^t \|\nabla(\sigma - \tilde{\sigma})^+(\tau)\|^2 \d \tau \leq \frac{1}{2}\|(\sigma_0 - \tilde{\sigma})^+\|^2 \no \\
&- \int_0^t \int_{\Omega} \Big \{\gamma_c \left(\sigma -\frac{\S_c}{\gamma_c}\right) \phi + \gamma_h \left( (\sigma -\frac{\S_h}{\gamma_h}\right)(1-\phi) \Big \} (\sigma - \tilde{\sigma})^+(\tau) \d x \d \tau \no \\
& - \int_0^t \int_{\Omega} s \phi (\sigma - \tilde{\sigma})^+(\tau) \d x \d \tau. 
\end{align}
Using $(\sigma -\frac{\S_h}{\gamma_h}) \geq (\sigma - \tilde{\sigma})$, $(\sigma -\frac{\S_c}{\gamma_c}) \geq (\sigma - \tilde{\sigma}),$ and $0\leq \phi \leq 1$, we further get
\begin{align}
\frac{1}{2}\|(\sigma - \tilde{\sigma})^+(t)\|^2 &+ \int_0^t \|\nabla(\sigma - \tilde{\sigma})^+(\tau)\|^2 \d \tau \leq \frac{1}{2}\|(\sigma_0 - \tilde{\sigma})^+\|^2 \no \\
&- \int_0^t \int_{\Omega} \{ \gamma_c \left(\sigma -\tilde{\sigma}\right) \phi + \gamma_h \left(\sigma -\tilde{\sigma}\right) (1-\phi) \} (\sigma - \tilde{\sigma})^+(\tau) \d x \d \tau \no \\
& - \int_0^t \int_{\Omega} s \phi (\sigma - \tilde{\sigma})^+(\tau) \d x \d \tau.
\end{align}
Using the fact that $(\sigma - \tilde{\sigma}) \leq (\sigma - \tilde{\sigma})^+$ we deduce
\begin{align}
\frac{1}{2}\|(\sigma - \tilde{\sigma})^+(t)\|^2 &+ \int_0^t \|\nabla(\sigma - \tilde{\sigma})^+(\tau)\|^2 \d \tau \leq \frac{1}{2}\|(\sigma_0 - \tilde{\sigma})^+\|^2 \no \\
&- \int_0^t \int_{\Omega} \left( \phi \gamma_c + (1-\phi) \gamma_h \right) (\sigma - \tilde{\sigma})^2(\tau) \d x \d \tau  - \int_0^t \int_{\Omega} s \phi (\sigma - \tilde{\sigma})^+(\tau) \d x \d \tau.
\end{align}
Since $(\sigma_0 - \tilde{\sigma})^+=0$ and the second and third term in the right hand side are nonpositive so we get $(\sigma - \tilde{\sigma})^+ =0$. So we showed that $0 \leq \sigma \leq \tilde{\sigma}=\max\{\frac{\S_h}{\gamma_h},\frac{\S_c}{\gamma_c}\} $ a.e. in $\Omega$, for all $t \in [0,T].$
\end{proof} 

\subsection{Attractor}
Let us introduce the phase space as following 
 $$\tilde{X}= \{(\phi,\sigma,p) \in L^2(\Omega) \times L^2(\Omega) \times L^2(\Omega): 0\leq \phi \leq 1, 0\leq \sigma \leq \tilde{\sigma}, p\geq 0 \},$$ where $\tilde{\sigma}>0$ is defined in Theorem \ref{sigmabound}. Now the space $\tilde{X}$ is complete metric space with the metric  associated with the distance
 $$d(\phi_1,\phi_2,\sigma_1,\sigma_2,p_1,p_2)= \|\phi_1-\phi_2\|_H + \|\sigma_1-\sigma_2\|_H + \|p_1-p_2\|_H.$$
Now we can introduce the solution operator 
$$S(t) : \tilde{X} \rightarrow \tilde{X}, \quad t \geq 0,$$
$$(\phi_0,\sigma_0,p_0) \rightarrow S(t)(\phi_0,\sigma_0,p_0) = (\phi(t),\sigma(t),p(t)).$$
It can be easily seen that $(S(t),d_{\tilde{X}})$ defines a strongly continuous semigroup.  

Now we will show that the dynamical system $(S(t),d_{\tilde{X}})$ associated to our problem poses a bounded absorbing set. We will use $C$ as a generic constant which is independent of time and initial conditions. This constant may vary even within the same line.
\begin{theorem}\label{thm3.4}
Let the assumption \ref{ass 1} be satisfied. Then $S(t)$ has a bounded absorbing set as following 
$$\mathcal{B}_0 := \{(\phi,\sigma,p) \in \tilde{X} : \|(\phi,\sigma,p)\| \leq C_0 \},$$
where $C_0$ is a positive constant that depends on the $H$ norm of the initial data and the constants $\gamma_h$, $\gamma_p$. That means for each bounded set $\mathcal{B}$ in $\tilde{X}$ there exist a time $t_0= t_0(\mathcal{B})>0$ such that any weak solution satisfies 
\begin{align}
\|(\phi(t),\sigma(t),p(t))\|_{\tilde{X}} \leq C_0 \quad \forall t \geq t_0.
\end{align}
\end{theorem}
\begin{proof}
 Taking inner product of \eqref{phi1} with $\phi$, \eqref{sigma1} with $\sigma$ and \eqref{p1} with $p$ and sum up to obtain
\begin{align}
  \frac{1}{2}   \frac{\d}{\d t}\left(\|\phi\|^2+\|\sigma\|^2+\|p\|^2\right) &+ \gamma_h\|\sigma\|^2+ \gamma_p\|p\|^2+ \lambda \|\nabla \phi\|^2 + \eta \|\nabla \sigma\|^2 +  D  \|\nabla p\|^2 \no \\ &\leq  -\int_{\Omega} 2\phi^2 (1-\phi)f(\phi,\sigma,u) \d x 
    - \int_{\Omega} \gamma_{ch} \sigma^2 \phi \d x \no \\ & + \int_{\Omega} \left( S_h  (1-\phi) + (S_c-s) \phi \right) \sigma \d x \no \\
    &+ \int_{\Omega} \{\alpha_h(1-\phi)+\alpha_c \phi\}p \d x. 
\end{align}
Now using the fact that $0<\phi<1$ and $0<\sigma<\tilde{\sigma}$ we get
\begin{align}
    \frac{1}{2} \frac{\d}{\d t}\left(\|\phi\|^2+\|\sigma\|^2+\|p\|^2\right) &+ \gamma_h\|\sigma\|^2+ \gamma_p\|p\|^2+ \lambda \|\nabla \phi\|^2 + \eta \|\nabla \sigma\|^2 +  D  \|\nabla p\|^2  \no  \\
     &\leq 2\|f\|_{\mathrm{L}^\infty}+ |\gamma_{ch}| \tilde{\sigma}^2 + C \tilde{\sigma} +\frac{ \gamma_p}{2} \|p\|^2 + C  \no
     \end{align}
     and therefore
     \begin{align}\label{e6}
   \frac{1}{2}   \frac{\d}{\d t}\left(\|\phi\|^2+\|\sigma\|^2+\|p\|^2\right) &+ \gamma_h\|\sigma\|^2+ \frac{ \gamma_p}{2}\|p\|^2+ \lambda \| \nabla \phi\|^2 + \eta \|\nabla \sigma\|^2 +  D  \|\nabla p\|^2  \leq C.
\end{align}
Now integrating over $0$ to $t$  and using Gronwall's lemma, we get for every $t>0$,
\begin{align*}
 \|\phi(t)\|^2 +\|\sigma(t)\|^2+\|p(t)\|^2 &+ \lambda \int_0^t \| \nabla \phi\|^2 + \eta \int_0^t \|\nabla \sigma\|^2 +  D \int_0^t \|\nabla p\|^2  \no \\
 &\leq (\|\phi_0\|^2+\|\sigma_0\|^2+\|p_0\|^2) e^{-\min\{2\gamma_h,\gamma_p \}t} + \frac{\bar{C}}{\min\{2\gamma_h,\gamma_p \}} .
\end{align*}
Now if we take $ t_0=\frac{1}{\min\{2\gamma_h,\gamma_p \}} \log{\frac{(\|\phi_0\|^2+\|\sigma_0\|^2+\|p_0\|^2)\bar{C}}{\min\{2\gamma_h,\gamma_p \}}\bar{C}}$, then we will  get from \eqref{e6}
\begin{align}\label{e7}
\|\phi(t)\|^2 +\|\sigma(t)\|^2+\|p(t)\|^2 \leq C_0, \quad \forall t \geq t_0,
\end{align}
where $$C_0=  (\|\phi_0\|^2+\|\sigma_0\|^2+\|p_0\|^2) e^{-\min\{2\gamma_h,\gamma_p \}t_0} + \frac{\bar{C}}{\min\{2\gamma_h,\gamma_p \}}.$$
\end{proof}

In the next theorem, we prove that the dynamical system $S(t)$ has compact absorbing sets which will guarantee the existence of the global attractor.
\begin{theorem}
Let Assumptions \ref{ass 1} hold. Then the dynamical system $S(t)$ admits the global attractor $\mathcal{A}$. More precisely, $\mathcal{A}$ is a compact subset of $\tilde{X}$ which is bounded in $H^1_0(\Omega) \times H^1 (\Omega)\times H^1(\Omega)$ and uniformly attracts the trajectories emanating from any bounded set $\B \subset \tilde{X}$.  We prove in the sense that there exist a positive constant $C_1$ such that for any bounded set $\mathcal{B}$ there exist a time $t_1=t_1(\mathcal{B})>0$ such that
$$\|(\phi,\sigma,p)\|_{H^1_0(\Omega) \times H^1 (\Omega)\times H^1(\Omega)} \leq C_1, \quad \forall t \geq t_1.$$ 
\end{theorem}
\begin{proof}
Now let us take inner product of \eqref{phi1} with $-\Delta \phi$, \eqref{sigma1} with $-\Delta \sigma$ and \eqref{p1} with $-\Delta p$ and sum up to obtain
\begin{align}
  \frac{1}{2}   \frac{\d}{\d t}\left( \|\nabla \phi\|^2 +  \|\nabla \sigma\|^2 +  \|\nabla p\|^2\right) & + \lambda \|\Delta \phi\|^2+\eta \|\Delta \sigma\|^2+D \|\Delta p\|^2 - \gamma_p(p,\Delta p)  \no \\ &\leq \int_{\Omega} 2\phi (1-\phi)f(\phi,\sigma,u) \Delta \phi \d x  \no \\ 
  & +  \int_{\Omega} \{\gamma_c \left(\sigma -\frac{\S_c}{\gamma_c}\right) \phi + \gamma_h \left( (\sigma -\frac{\S_h}{\gamma_h}\right)(1-\phi) \} \Delta \sigma \d x \no \\
& +  \int_{\Omega} s \phi \Delta \sigma  \d x - \int_{\Omega} \{\alpha_h(1-\phi)+\alpha_c \phi\}\Delta p \d x .
\end{align}
Using Gagliardo-Nirenberg, Young’s inequalities and the boundedness of $\phi$ and $\sigma$ we  deduce
\begin{align}\label{e8}
\frac{1}{2}   \frac{\d}{\d t}\left( \|\nabla \phi\|^2 +  \|\nabla \sigma\|^2 +  \|\nabla p\|^2\right) & + \frac{\lambda}{2} \|\Delta \phi\|^2+\frac{\eta}{2} \|\Delta \sigma\|^2+\frac{D}{2} \|\Delta p\|^2 + \gamma_p \|\nabla p\|^2  \no \\ &\leq \frac{2}{\lambda}\int_{\Omega}\|f\|_{L^\infty}^2  \d x  \no \\ 
  & + \frac{1}{\eta} \int_{\Omega} \{\gamma_c \left(\sigma -\frac{\S_c}{\gamma_c}\right) \phi + \gamma_h \left( (\sigma -\frac{\S_h}{\gamma_h}\right)(1-\phi) \}^2 \d x \no \\
& +  \int_\Omega \phi s \Delta \sigma  + \frac{1}{2D} \int_{\Omega} \{\alpha_h(1-\phi)+\alpha_c \phi\}^2 \d x \no \\
& \leq  C + \frac{1}{\eta}\|s\|_{L^\infty}^2 \leq C
\end{align}
Now using Theorem \ref{thm3.4} and Gronwall's Lemma we obtain 
\begin{align}\label{e9}
 \|\nabla \phi(t)\|^2 +  \|\nabla \sigma(t)\|^2 +  \|\nabla p(t)\|^2 \leq Ce^{\gamma_p}, \text{ for all }t \geq t_1=t_0+1.
\end{align}
Adding \eqref{e7} and \eqref{e9} we get for all $t \geq t_1$
\begin{align}\label{e10a}
\|\phi(t)\|_{V_0}^2+\|\sigma(t)\|_V^2+\|p(t)\|_V^2 \leq C_1.  
\end{align}
Therefore from \eqref{e7} and \eqref{e10a} we get a global attractor $\mathcal{A}$ as well as its boundedness in $V_0 \times V \times V$.
\end{proof}

Now we prove a lemma before proving the main result.
\begin{lemma}\label{lem3.6}
Let $Z$ be a Banach space. Let us take $f \in L^2(0,\infty;Z)$ such that for some $p\in (1,\infty)$ we have
\begin{align}
f \in  W^{1,p}(0,t;Z) \ \ \forall t \in [0,\infty)  \text{  and  } \sup_{t \in [0,\infty)} \|f_t\|_{L^p(t,t+1;Z)} < \infty.
\end{align}
Then $f \in L^\infty(0,\infty;Z)$. Moreover, if $p>1$, then $f(t) \rightarrow 0$ strongly in $Z$ as $t \rightarrow \infty$. 
\end{lemma}

\begin{theorem}
Let us take $\phi_0, \sigma_0, p_0 \in H$, then we have
\begin{align}
\phi \in  W^{1,2}(0,t;H) \ \ \forall t \in [0,\infty)  \text{  and  } \sup_{t \in [0,\infty)} \|\phi_t\|_{L^2(t,t+1;H)} < \infty.
\end{align}
Therefore we get from the previous Lemma \ref{lem3.6} that $\phi(t) \rightarrow 0$ as $t \rightarrow \infty$. 
\end{theorem}
\begin{proof}
Let us take inner product of \eqref{phi1} with $ \phi_t$, \eqref{sigma1} with $\sigma_t$ and \eqref{p1} with $p_t$ and sum up to obtain
\begin{align}\label{e7a}
   \|\phi_t\|^2 +  \|\sigma_t\|^2 +  \|p_t\|^2 & + \frac{1}{2}   \frac{\d}{\d t}\left(\lambda \| \nabla \phi\|^2+ \min \{\eta,\gamma_h\} \|\sigma\|_V^2+ \min \{D,\gamma_p\} \| p\|_V^2 \right)  \no \\ &\leq \int_{\Omega} 2\phi (1-\phi)f(\phi,\sigma,u) \phi_t \d x  \no \\ 
  & +  \int_{\Omega} \{\gamma_c \left(\sigma -\frac{\S_c}{\gamma_c}\right) \phi + \gamma_h \left( \sigma -\frac{\S_h}{\gamma_h}\right)(1-\phi) \} \sigma_t \d x \no \\
& -  \int_{\Omega} s \phi \sigma_t  \d x - \int_{\Omega} \{\alpha_h(1-\phi)+\alpha_c \phi\} p_t \d x  \no \\ &\leq 2  \|f\|_{L^\infty}^2  + \frac{\|\phi_t\|^2}{2}   \no \\ 
  & + \frac{1}{4} \int_{\Omega} \{\gamma_c \left(\sigma -\frac{\S_c}{\gamma_c}\right) \phi + \gamma_h \left( (\sigma -\frac{\S_h}{\gamma_h}\right)(1-\phi) \}^2 \d x  + \frac{\|\sigma_t\|^2}{2}\no \\
& + \|s\|_{L^\infty}^2 + \frac{1}{2} \int_{\Omega} \{\alpha_h(1-\phi)+\alpha_c \phi\}^2 \d x + + \frac{\|p_t\|^2}{2} \leq C .
\end{align}
Therfore after integrating we get
$\phi_t \in L^2(0,t;H)$ for all $t\in [0,T]$.
We finally get $\phi \in W^{1,2}(0,t;H)$. Using the Lemma \ref{lem3.6} we get $\phi(t) \rightarrow 0$ as $t$ goes to $\infty$.
\end{proof}

\section{Convergence to equilibrium}
One can observe from \cite{LG} that, all the estimates proved in the paper are uniform over the given time interval $(0,T)$, where $T>0$. Hence, standard extension arguments imply that, after taking the limit with respect to the approximation parameter, we obtain a global in time solution.

Thus we can study now the long time behavior of the system \eqref{phi1}-\eqref{initial1} as $t \rightarrow \infty$. We will show that a weak solution of \eqref{phi1}-\eqref{initial1} converges in a proper sense to a strong solution of the following steady state system. The system is given by
\begin{subequations}
		\begin{align}
	             - \eta \Delta \sigma_\infty + \gamma_h \sigma_\infty  &= S_h , \label{ssigma}\ \text{ in }\ \Omega,\\
		- D \Delta p_\infty + \gamma_p p_\infty  &= \alpha_h , \label{sp}\ \text{ in }\ \Omega,\\
		\phi_\infty &=0 , \ \text{ in }\ \Omega,\\
		\frac{\partial \sigma_\infty}{\partial \nu} = \frac{\partial p_\infty}{\partial \nu} &= 0, \text{ on }\ \partial\Omega. \label{sboundary}
		\end{align}   
	\end{subequations}
	
	\begin{definition}\label{weak1}
	A weak solution to the system \eqref{ssigma}-\eqref{sboundary} is $(\sigma_\infty,p_\infty)$ with
	\begin{align}
	&\sigma_{\infty} \in V \no  \\
	& p_{\infty} \in V,
	\end{align}
	and it satisfies
	\begin{align}\label{e10b}
(\eta \nabla \sigma_\infty,\nabla u) + (\gamma_h \sigma_\infty,u) &= (S_h,u), \ \forall u\in V, \no \\
 (D \nabla p_\infty ,\nabla v)+ (\gamma_p p_\infty,v)  &=  (\alpha_h,v), \ \forall v\in V.
\end{align} 
	\end{definition}
	\begin{definition}\label{strong}
	A weak solution to  \eqref{ssigma}-\eqref{sboundary} is called a strong solution if $(\sigma_\infty,p_\infty) \in H^2(\Omega) \times H^2(\Omega)$.
	\end{definition}

First we show that the stationary problem has at least a strong solution. Let us define the functional 
\begin{equation}\label{func}
\Gamma(u,v)= \int_\Omega \left(\frac{\eta}{2}|\nabla u|^2 + \frac{\gamma_h}{2} |u|^2 - S_h u\right) \d x + \int_\Omega \left(\frac{D}{2}|\nabla v|^2 + \frac{\gamma_p}{2} |v|^2- \alpha_h v \right) \d x,
\end{equation}
for any $u,v \in V.$
\begin{theorem}\label{exist}
Let $(\sigma_\infty,p_\infty) \in H^2(\Omega) \times H^2(\Omega)$ be a strong solution of the system \eqref{ssigma}-\eqref{sboundary}, then $(\sigma_\infty,p_\infty)$ is a critical point of the functional \eqref{func} over $V \times V$. Conversely, if $(\sigma_\infty,p_\infty)$ is a critical point of the functional \eqref{func} over $V \times V$, then $(\sigma_\infty,p_\infty) \in H^2(\Omega) \times H^2(\Omega)$ and it is a strong solution of the system \eqref{ssigma}-\eqref{sboundary}.
\end{theorem}
\begin{proof}
Since $(\sigma_\infty,p_\infty) \in H^2(\Omega) \times H^2(\Omega)$ be a strong solution of the system \eqref{ssigma}-\eqref{sboundary}, therefore we get from the weak formulation \eqref{e10b} and integrating by parts that
\begin{align}
(\eta \nabla \sigma_\infty,\nabla u) + (\gamma_h \sigma_\infty,u) + (D \nabla p_\infty ,\nabla v)+ (\gamma_p p_\infty,v)  = (S_h,u)+ (\alpha_h,v),
\end{align} 
which, by a straightforward calculation, can be written as
\begin{align}\label{e10}
\lim_{\epsilon \rightarrow 0} \frac{d \Gamma(\sigma_\infty +\epsilon u,p_\infty+ \epsilon v)}{d \epsilon} = 0.
\end{align}
Therefore we get from \eqref{e10} that $(\sigma_\infty,p_\infty)$ is a critical point of the functional \eqref{func}.

Conversely let us take $(\sigma_\infty,p_\infty)$ is a critical point of the functional \eqref{func}.
Similarly by taking the Gateaux derivative of \eqref{func} in the direction of $(u,v)$ i.e.
\begin{align}\label{e11}
\lim_{\epsilon \rightarrow 0} \frac{\Gamma(\sigma_\infty +\epsilon u,p_\infty+ \epsilon v) - \Gamma(\sigma_\infty ,p_\infty)}{\epsilon}= 0.
\end{align}
From \eqref{e11}, using integration by parts we get we get for all $u,v \in V$
\begin{align}\label{e12a}
\epsilon(\eta \nabla \sigma_\infty,\nabla u) &+ \frac{\eta}{2}\epsilon^2 \|\nabla u\|^2 + \epsilon(\gamma_h \sigma_\infty,u) + \frac{\eta_h}{2}\epsilon^2 \| u\|^2+\epsilon (D \nabla p_\infty ,\nabla v)+ \frac{D}{2} \epsilon^2 \|\nabla v\|^2 + \epsilon(\gamma_p p_\infty,v) \no \\
&+ \frac{\gamma_p}{2} \epsilon^2 \|v\|^2 - \epsilon(S_h,u)- \epsilon(\alpha_h,v)=0.
\end{align}
Now dividing \eqref{e12a} by $\epsilon$ and taking $\epsilon$ to $0$ we finally get for all $u,v \in V$
\begin{align}\label{e13a}
(\eta \nabla \sigma_\infty,\nabla u) + (\gamma_h \sigma_\infty,u) + (D \nabla p_\infty ,\nabla v)+ (\gamma_p p_\infty,v)= (S_h,u)+ (\alpha_h,v).
\end{align}
From \eqref{e13a} we can say that $(\sigma_\infty,p_\infty)$ is the weak solution of \eqref{ssigma}-\eqref{sboundary}. By the elliptic regularity theory we have $(\sigma_\infty,p_\infty) \in H^2(\Omega) \times H^2(\Omega)$.
\end{proof}

In the the next theorem we prove existence of the steady state solution. 
\begin{theorem}
The functional $\Gamma$ has a minimizer $(\sigma_\infty,p_\infty)$ i.e. the problem 
\begin{align}
\inf_{(u,v) \in V \times V} \Gamma(u,v)
\end{align}
has one solution $(\sigma_\infty,p_\infty)$. Hence there exist a strong solution for the steady state system \eqref{ssigma}-\eqref{sboundary}. 
\end{theorem}
\begin{proof}
The functional $\Gamma$ is bounded from below. Therefore there exist a minimizing sequence $(u_n,v_n) \in V \times V$ such that
\begin{align}\label{e12}
\lim_{n \rightarrow \infty} \Gamma(u_n,v_n) = \inf_{(u,v) \in V \times V} \Gamma(u,v).
\end{align} 
Claim 1- We need to show $(u_n,v_n)$ is bounded in $V \times V$. 

From \eqref{func} we derive that
\begin{align}\label{e13}
\Gamma(u_n,v_n)&= \int_\Omega \left(|\nabla u_n|^2 + \frac{\gamma_h}{2} |u_n|^2 - S_h u_n\right) \d x + \int_\Omega \left(|\nabla v_n|^2 + \frac{\gamma_p}{2} |v_n|^2- \alpha_h v_n \right) \d x \no \\
& \quad \geq \|\nabla u_n\|^2 +\frac{\gamma_h}{2} \|u_n\|^2 - S_h \|u_n\|^2 + \|\nabla v_n\|^2 + \frac{\gamma_p}{2} \|v_n\|^2- \alpha_h \|v_n\|^2  \no \\
& \quad \geq \min\{ 1, \left(\frac{\gamma_h}{2} -S_h\right)\} \|u_n\|^2_{H^1(\Omega)} + \min\{ 1, \left(\frac{\gamma_p}{2} -\alpha_h\right)\} \|v_n\|^2_{H^1(\Omega)} \no \\
& \quad \geq C\left( \|u_n\|^2_{H^1(\Omega)} + \|v_n\|^2_{H^1(\Omega)} \right).
\end{align}
We have from \eqref{e12} that the convergent sequence $\Gamma(u_n,v_n)$ is bounded above. Thus it follows from \eqref{e13} that the sequence $(u_n,v_n)$ is bounded in $V \times V$. Then we will get convergent subsequence still denoted by $(u_n,v_n)$. Let us take the limit $(\sigma_\infty,p_\infty)$.

Claim 2- $(\sigma_\infty,p_\infty)$ is a minimizer of the functional $\Gamma$. 

Since the functional $\Gamma$ is continuous and convex. Therefore $\Gamma$ is weakly lower semicontinuous. So For a weakly convergent sequence i.e.
\begin{align}
(u_n,v_n) \rightharpoonup (\sigma_\infty,p_\infty),
\end{align}
we have
\begin{align}
\Gamma(\sigma_\infty,p_\infty) \leq \liminf \Gamma(u_n,v_n).
\end{align}
Therefore we get 
\begin{align}
\inf_{(u,v) \in V \times V} \Gamma(u,v) \leq \Gamma(\sigma_\infty,p_\infty)  \leq \liminf \Gamma(u_n,v_n) \leq \lim \Gamma(u_n,v_n) \leq \inf_{(u,v) \in V \times V} \Gamma(u,v) .
\end{align}
It proves that $(\sigma_\infty,p_\infty)$ is a minimizer of the functional $\Gamma$.
\end{proof}

Now we prove that the long time convergence result in the following theorem.
\begin{theorem}\label{main}
For any $\phi_0 \in L^2(\Omega)$, $\sigma_0 \in L^2(\Omega)$ and $p_0 \in L^2(\Omega)$, the system \eqref{phi1}-\eqref{initial1} admits a unique global weak solution $(\phi,\sigma,p)$ such that 
\begin{eqnarray}\label{e49}
\lim_{t \rightarrow \infty} \left( \|\phi(t)\|_{L^2(\Omega)} + \|\sigma(t)-\sigma_\infty\|_{L^2(\Omega)} + \|p(t)-p_\infty\|_{L^2(\Omega)} \right) =0,
\end{eqnarray}
where $(\sigma_\infty,p_\infty)$ satisfies \eqref{e10b}.

We show that if $\lambda\lambda_1 \geq \frac{|\gamma_{ch}|^2|\sigma_\infty|^2}{\gamma_h} + \frac{|\alpha_{ch}|^2}{2\gamma_p}+2\|f\|_{L^\infty}$, then there exist a constant $\beta >0$ and $C >0$ such that
\begin{align}
\|\phi(t)\|_{L^2(\Omega)} + \|\sigma(t)-\sigma_\infty\|_{L^2(\Omega)} + \|p(t)-p_\infty\|_{L^2(\Omega)} \leq C e^{-\beta t},
\end{align}
where $\lambda_1$ is defined in \eqref{lambda} and $C= \|\phi_0\|_{L^2(\Omega)} + \|\sigma_0-\sigma_\infty\|_{L^2(\Omega)} + \|p_0-p_\infty\|_{L^2(\Omega)}$ and
 
$\beta = \min\{\left(\lambda\lambda_1 - \frac{|\gamma_{ch}|^2|\sigma_\infty|^2}{\gamma_h} - \frac{|\alpha_{ch}|^2}{2\gamma_p}-2\|f\|_{L^\infty} \right),\frac{\gamma_p}{2},\frac{\gamma_h}{2}\}.$
\end{theorem}

\begin{proof}Let us denote $w(t)=\phi(t)$, $y(t)=\sigma(t)-\sigma_\infty$ and $z(t)=p(t)-p_\infty$. The system satisfied by $(w,y,z)$ is 
\begin{subequations}
		\begin{align}
		w_t -\lambda \Delta w + 2w (1-w) f(w,y+\sigma_\infty,u) &= 0, \label{w}\ \text{ in }\ Q_T:=(0,\infty)\times\Omega,\\
				y_t - \eta \Delta y + \gamma_h y + \gamma_{ch}yw+ \gamma_{ch} \sigma_\infty w &=  (S_c-S_h)w - sw, \label{y}\ \text{ in }\ Q_T,\\
		z_t - D \Delta z + \gamma_p z  &= (\alpha_c-\alpha_h)w, \label{z}\ \text{ in }\ Q_T,\\
		w = 0, \ \frac{\partial y}{\partial \nu} &= \frac{\partial z}{\partial \nu} = 0, \text{ on }\ \Sigma_T:=(0,\infty)\times\partial\Omega, \label{boundary1}\\
		w(0) &= w_0, \  \ y(0) =  y_0, \ z(0)=z_0, \text{ in } \ \Omega. \label{initial2}
		\end{align}   
	\end{subequations}
	
	Taking inner product of \eqref{w} with $w$, \eqref{y} with $y$ and \eqref{z} with $z$ we get
	\begin{align}\label{e40}
	\frac{\d}{\d t} &\left( \|w\|^2 + \|y\|^2 + \|z\|^2 \right) + \lambda \|\nabla w\|^2 + \eta \|\nabla y\|^2 + D \|\nabla z\|^2 + \gamma_h \|y\|^2 + \gamma_p \|z\|^2 \nonumber \\
	 &= -\int_\Omega 2w^2(1-w) f(w,y+\sigma_\infty,u) - \int_\Omega \gamma_{ch}y^2 w - \int_\Omega \gamma_{ch} \sigma_\infty wy \nonumber \\
	 &+ \int_\Omega (S_{ch}- s)wy + \int_\Omega \alpha_{ch} wz \nonumber \\
	 &= \sum_{i=1}^5 I_i
	\end{align} 
	Now we will estimate the r.h.s of \eqref{e40} by using H\"older's, Poincar\'e inequality and Young's inequalities and derive
	\begin{align}\label{e41}
	|I_1| \leq \Big|\int_\Omega 2w^2(1-w) f(w,y+\sigma_\infty,u) \d x \Big| \leq 2\|f\|_{\infty} \|w\|^2,
	\end{align}
	\begin{align}\label{e42}
	|I_2| \leq \Big|\int_\Omega \gamma_{ch}y^2 w \d x \Big| \leq |\gamma_{ch}|\|y\|^2,
	\end{align}
	\begin{align}\label{e43}
	|I_3| \leq \Big|\int_\Omega \gamma_{ch} \sigma_\infty wy \d x \Big| \leq |\gamma_{ch}||\sigma_\infty||w||y| \leq \frac{|\gamma_{ch}|^2|\sigma_\infty|^2}{\gamma_h}\|w\|^2 + \frac{\gamma_h}{4}\|y\|^2,
	\end{align}
	\begin{align}\label{e44}
	|I_4| \leq \Big|\int_\Omega (S_{ch}- s)wy \d x \Big| \leq (|S_{ch}|+\|s\|_{L^{\infty}}|)|w||y| \leq \frac{2(|S_{ch}|^2+\|s\|_{L^{\infty}}^2)}{\gamma_h}\|w\|^2 + \frac{\gamma_h}{4}\|y\|^2,
	\end{align}
	\begin{align}\label{e45}
	|I_5| \leq \Big|\int_\Omega \alpha_{ch} wz \d x \Big| \leq |\alpha_{ch}||w||z| \leq \frac{|\alpha_{ch}|^2}{2\gamma_p}\|w\|^2 + \frac{\gamma_p}{2}\|z\|^2.
	\end{align}
	Now adding \eqref{e41}-\eqref{e45} and substituting in \eqref{e40} we get
	\begin{align}\label{e46}
	\frac{\d}{\d t} &\left( \|w\|^2 + \|y\|^2 + \|z\|^2 \right) + \lambda \|\nabla w\|^2 + \eta \|\nabla y\|^2 + D \|\nabla z\|^2 + \frac{\gamma_h}{2} \|y\|^2 + \frac{\gamma_p}{2} \|z\|^2 \nonumber \\
	 &\leq 2\|f\|_{L^\infty} \|w\|^2 + |\gamma_{ch}|\|y\|^2 + \frac{|\gamma_{ch}|^2|\sigma_\infty|^2}{\gamma_h}\|w\|^2 + \frac{|\alpha_{ch}|^2}{2\gamma_p}\|w\|^2.
	 \end{align}
	 By using Poincar\'e inequality we further derive
	 \begin{align}\label{e47}
	\frac{\d}{\d t} &\left( \|w\|^2 + \|y\|^2 + \|z\|^2 \right) + \left(\lambda\lambda_1 - \frac{|\gamma_{ch}|^2|\sigma_\infty|^2}{\gamma_h} - \frac{|\alpha_{ch}|^2}{2\gamma_p} -2\|f\|_{L^\infty} \right) \|w\|^2 \nonumber \\
	&+ \eta \|\nabla y\|^2 + D \|\nabla z\|^2 + \frac{\gamma_h}{2} \|y\|^2 + \frac{\gamma_p}{2} \|z\|^2 \leq 0,
	 \end{align}
	 where $\lambda_1$ is defined in \eqref{lambda}.
	 Using Gronwall's Lemma we finally obtain
	 \begin{align}\label{e48}
	 &\|w\|^2 + \|y\|^2 + \|z\|^2  \leq \left( \|w_0\|^2 + \|y_0\|^2 + \|z_0\|^2 \right)e^{-\beta t} . 
	 \end{align}
	 where $\beta = \min\{\left(\lambda\lambda_1 - \frac{|\gamma_{ch}|^2|\sigma_\infty|^2}{\gamma_h} - \frac{|\alpha_{ch}|^2}{2\gamma_p}-2\|f\|_{L^\infty} \right),\frac{\gamma_p}{2},\frac{\gamma_h}{2}\}.$
	 
	 It gives the equality \eqref{e49} which completes the proof .
\end{proof}

\end{document}